\documentclass[11pt]{article}
\usepackage{color}
\usepackage[colorlinks=true]{hyperref}
\usepackage[margin=1in]{geometry}
\usepackage{amssymb}
\usepackage{amsmath}
\usepackage{amsthm}
\usepackage{mathtools}
\usepackage{epsfig, graphicx}
\usepackage{lipsum}
\usepackage{amsfonts}
\usepackage{graphicx}
\usepackage{epstopdf}
\usepackage[boxed, noline, ruled, linesnumbered]{algorithm2e}
\usepackage{algorithmic}
\usepackage{tikz}
\usepackage{stmaryrd}
\usepackage{cleveref}
\newtheorem{theorem}{Theorem}[section]
\newtheorem{lemma}{Lemma}[section]

\newtheorem{remark}{Remark}[section]

\makeatletter % '@' is now a normal "letter" for TeX

\@addtoreset{equation}{section}
\makeatother % '@' is restored as a "non-letter" character for TeX
\allowdisplaybreaks % For long formula
\newcommand{\comm}[1]{{\color{black}#1}} % red
\newcommand{\revise}[1]{{\color{black}#1}}  %blue
\begin{document}
\title{Neural Network Element Method for Partial Differential Equations\footnote{This work was supported by 
the Strategic Priority Research Program of the Chinese Academy of 
Sciences (XDB0620203, XDB0640000, XDB0640300),  National Key Laboratory of Computational Physics 
(No. 6142A05230501), National Natural Science Foundations of 
China (NSFC 1233000214), National Center for Mathematics and Interdisciplinary Science, CAS.}}
\author{ Yifan Wang\footnote{School of Mathematical Sciences, Peking University, 
Beijing 100871, China (wangyifan1994@pku.edu.cn).}, \ \ 
Zhongshuo Lin\footnote{LSEC, NCMIS, Institute
of Computational Mathematics, Academy of Mathematics and Systems
Science, Chinese Academy of Sciences, Beijing 100190,
China,  and School of Mathematical Sciences, University
of Chinese Academy of Sciences, Beijing 100049, China (linzhongshuo@lsec.cc.ac.cn).} %, \ \ 
%Yangfei Liao\footnote{LSEC, NCMIS, Institute
%of Computational Mathematics, Academy of Mathematics and Systems
%Science, Chinese Academy of Sciences, Beijing 100190,
%China,  and School of Mathematical Sciences, University
%of Chinese Academy of Sciences, Beijing 100049, China (liaoyangfei@lsec.cc.ac.cn).}, \ \ 
%Haochen Liu\footnote{LSEC, NCMIS, Institute
%of Computational Mathematics, Academy of Mathematics and Systems
%Science, Chinese Academy of Sciences, Beijing 100190,
%China,  and School of Mathematical Sciences, University
%of Chinese Academy of Sciences, Beijing 100049, China (liuhaochen@lsec.cc.ac.cn).}
\ \ \  and \ \
Hehu Xie\footnote{LSEC, NCMIS, Institute
of Computational Mathematics, Academy of Mathematics and Systems
Science, Chinese Academy of Sciences, Beijing 100190,
China,  and School of Mathematical Sciences, University
of Chinese Academy of Sciences, Beijing 100049, China (hhxie@lsec.cc.ac.cn).}}

\date{}
\maketitle

\begin{abstract}
In this paper, based on the combination of finite element mesh and neural network, 
a novel type of neural network element space and corresponding machine 
learning method are designed for solving partial differential equations. 
The application of finite element mesh makes the neural network element space 
satisfy the boundary value conditions directly on the complex geometric domains.  
The use of neural networks allows the accuracy of the approximate solution to 
reach the high level of neural network approximation %without depending on the mesh 
even for the problems with singularities. 
\revise{We also provide the error analysis of the proposed method for the understanding.}  
The proposed numerical method in this paper provides the way to enable neural network-based 
machine learning algorithms to solve a broader range of problems arising 
from engineering applications.

%-----------------------------------------------------------------------------------------------------
\vskip0.3cm {\bf Keywords.} Neural network element, finite element mesh,  machine learning, 
partial differential equation, boundary value condition, complex geometry, singularity. 
%-----------------------------------------------------------------------------------------------------
\vskip0.2cm {\bf AMS subject classifications.}  68T07, 65L70, 65N25, 65B99.
\end{abstract}

\section{Introduction}
It is well known that solving partial differential equations (PDEs) is one of the 
most essential tasks in modern science and engineering society \cite{Evans}. 
There have developed many successful numerical methods such as finite difference, 
finite element, and spectral method for solving PDEs in three spatial dimensions 
plus the temporal dimension. Among these numerical methods, 
the Finite Element Method (FEM) is a powerful and widely used numerical technique for 
solving a variety of PDEs that arise in engineering, 
physics, and applied mathematics. By breaking down complex problems into simpler, 
smaller subdomains, FEM provides a flexible and efficient way to approximate 
solutions for problems that may be difficult or impossible to 
solve analytically \cite{brenner2007mathematical,ciarlet2002finite}. 

At its core, FEM involves discretizing a domain into a finite number of smaller, 
non-overlapping subdomains called elements, which collectively form a mesh. 
Each element is typically associated with a set of nodes, 
and approximate solutions are expressed as combinations of basis functions defined 
over these elements. These basis functions are often polynomials, chosen for 
their mathematical properties and computational efficiency. 
The key steps in the FEM include:
\begin{enumerate}
\item 
Problem Formulation: The problem is first expressed in a weak (or variational) form, 
which involves integrating the PDE against a set of test functions. 
This weak formulation ensures that solutions are well-posed for irregular 
geometries and boundary conditions.

\item 
Mesh Generation: The domain is divided into elements using a mesh, which can be 
uniform or non-uniform, depending on the geometry and solution requirements.

\item 
Basis Function Selection: Basis functions, typically piecewise linear or higher-order 
polynomials, are chosen to approximate the solution locally within each element.

\item 
Assembly: The global system of equations is assembled by combining the contributions 
from individual elements, ensuring continuity across element boundaries and applying 
boundary conditions.

\item 
Solution: The resulting system of algebraic equations is solved using numerical 
techniques to compute the approximate solution over the entire domain.
\end{enumerate}
The FEM is renowned for its versatility and adaptability. It can handle problems 
involving complex geometries, non-homogeneous materials, and arbitrary boundary conditions. 
It is particularly effective in areas like structural analysis, fluid dynamics, 
heat transfer, and electromagnetics. Moreover, its ability to adapt the mesh and 
basis functions to regions of high error allows for efficient and accurate solutions, 
even for problems with singularities or sharp gradients.

The FEM offers several advantages over the Finite Difference Method (FDM), 
making it a preferred choice in many engineering and scientific applications. 
Even FDM is simpler to implement and often computationally 
faster for simple, regular problems, the FEM excels in versatility, 
accuracy, and adaptability, especially for problems with complex geometries, 
boundary conditions, singularity, and material properties. These advantages make FEM a robust 
and widely adopted tool in many fields. Despite its strengths, FEM has challenges, 
including the need for high computational resources for large or three-dimensional 
problems and careful handling of mesh generation to ensure accuracy and stability. 
Nevertheless, it remains a cornerstone of numerical simulation and modeling 
in both academia and industry.

%For instance, Lagaris et al. \cite{Lagaris19981998} are early pioneers in applying 
%artificial neural networks to solve differential equations, constructing approximate 
%solutions for initial and boundary value problems. 
%Raissi et al. \cite{Raissi2019} introduce the physics-informed neural networks (PINNs), 

Due to its universal approximation property, the fully connected neural network (FNN) 
is the most widely used architecture to build the functions for solving PDEs \cite{LagarisLikasFotiadis}. 
There have developed several types of well-known FNN-based methods 
such as deep Ritz \cite{EYu}, deep Galerkin method \cite{DGM}, 
PINN \cite{RaissiPerdikarisKarniadakis}, and weak adversarial networks \cite{WAN}
for solving PDEs by designing different type of loss functions.  
Among these methods, the loss functions always include computing 
integration for the functions defined by FNN. 
For example, the loss functions of the deep Ritz method require computing 
the integration on the domain for the functions constructed by FNN.  
Always, these integration is computed using the Monte-Carlo 
method along with some sampling tricks \cite{EYu,HanZhangE}. 
Due to the low convergence rate of the Monte-Carlo method, the solutions obtained 
by the FNN-based numerical methods are challenging to achieve high accuracy and 
stable convergence process. This means that the Monte-Carlo method decreases 
computational work in each forward propagation while decreasing the simulation 
accuracy, efficiency and stability of the FNN-based numerical methods 
for solving PDEs. When solving nonhomogeneous 
boundary value problems, it is difficult to choose the number 
of sampling points on the boundary and in the domain. 
Furthermore, for solving non-homogeneous Dirichlet boundary value problems, besides 
the difficulty of choosing sampling points, it is also very difficult to set the 
hyperparameter to balance the loss from the boundary and interior domain.

Recently, \cite{LinWangXie} gives an error analysis framework 
for the NN based machine learning method for solving PDEs.  
This paper reveals that the integration error also controls 
the accuracy of the machine learning methods. Based on this 
conclusion, in order to improve the accuracy of the machine learning methods, 
we should use the quadrature schemes with high accuracy and high efficiency. 
Then, the deduced machine learning method can achieve high accuracy in solving PDEs. 
Even for high dimensional PDEs, we propose a type of tensor neural network (TNN) and 
the corresponding machine learning method, aiming to solve high-dimensional 
problems with high accuracy \cite{WangJinXie,WangLiaoXie,WangXie}. 
The reason of high accuracy of TNN based machine learning method is 
that the integration of TNN functions can be separated into one-dimensional 
integrations which can be computed by classical quadratures with high accuracy.  
The TNN-based machine learning method has already been used to solve high-dimensional 
eigenvalue problems and boundary value problems based on the Ritz type of loss functions. 
Furthermore,  in \cite{WangXie},  the multi-eigenpairs can also be computed with 
the machine learning method designed by combining the TNN and Rayleigh-Ritz process. 
The TNN is also used to solve 20,000 dimensional Schr\"{o}dinger equation 
with coupled quantum harmonic oscillator potential function \cite{HuShuklaKarniadakisKawaguchi}, 
high-dimensional Fokker-Planck equations \cite{WangHuKawaguchiZhangKarniadakis} 
and high-dimensional time-dependent problems \cite{KaoZhaoZhang}.  
\revise{We should also mention the subspace type of machine learning method, such as 
Random NN (RNN) \cite{DongLi1,DongLi2,HuangZhuSiew,LiWang,ShangWangSun,SunDongWang}, 
Random Feature  Mapping (RFM) \cite{ChenChiEYang}, 
Subspace Neural Network (SNN) \cite{LiuXuSheng, XuSheng}. 
In \cite{LinWangXie}, we design a type of adaptive subspace method for solving 
PDEs with high accuracy.}

For simplicity and easy understanding, in this paper, we are concerned 
with the following seconde order elliptic problems: Find $u(x)\in H_0^1(\Omega)$ such that 
\begin{equation}\label{Elliptic_Problem}
\left\{
\begin{array}{rcll}
-\nabla\cdot(\mathcal A\nabla u(x))+  b(x)u(x)&=&f(x), &  {\rm in}\ \Omega,\\
u&=&0, & {\rm on}\ \partial\Omega,
\end{array}
\right.
\end{equation} 
where $\Omega\subset \mathbb R^d$ is a Lipschitz domain and $d\leq 3$ in this paper, 
$H_0^1(\Omega)$ denotes the Sobolev space, $\mathcal A \in \mathbb R^{d\times d}$ 
is a symmetric positive definite matrix and the function $b(x)$ has a positive lower bound.
%\comm{Here, we will propose the machine learning type of numerical methods 
%for solving this boundary value problems}. 

\revise{In this paper, we design a new type of neural network (NN) element method 
for solving (\ref{Elliptic_Problem}) by combining the finite element mesh, 
piecewise polynomials basis functions and neural network to build the trial function space. 
The method here can combine the ability of FEM for complex geometries and 
strong approximation of NN. Furthermore, the deduced machine learning method 
has the similar efficiency and stability to the moving mesh method \cite{HuangRussell,LiTangZhang}. 
We use the mesh and corresponding finite element basis to handle the boundary 
value condition and the NN to enhance the approximation ability of 
the finite element basis functions.}  %\comm{The reasonable combination of finite element 
%basis and NN enables the approximation accuracy to be independent of the mesh size.} 
The corresponding error analysis is also provided for understanding the proposed 
method here.

\revise{Different from common training methods, the training step here is decomposed into 
two substeps including the linear solving or least square step for the coefficient 
and the optimization step for updating the neural networks \cite{WangLinLiaoLiuXie}. 
This separation scheme obviously improves the accuracy of concerned machine learning method. 
For the non-homogeneous boundary value conditions, the way in \cite{WangLinLiaoLiuXie} 
can be adopted for non-penalty terms with high accuracy.}

\revise{An outline of the paper follows. In Section \ref{Section_NN_Element}, 
we introduce the way to build the NN element space and error estimates 
of the NN element approximation. Section \ref{Section_Machine_Learning} 
is devoted to proposing the NN element based machine learning method for solving PDEs. 
Section \ref{Section_Numerical} gives some numerical examples to validate the accuracy and 
efficiency of the proposed NN element based machine learning method. 
Some concluding remarks are given in the last section.}

\section{Neural network element space and its basis}\label{Section_NN_Element}
The finite element mesh divides the domain into smaller, non-overlapping subdomains called 
elements (e.g., triangles, quadrilaterals, tetrahedra, or hexahedra), 
which collectively approximate the geometry of the domain. The mesh is a 
discretized representation of a geometric domain used in FEM 
for solving PDEs. The finite element space is a type of piecewise polynomial 
defined on the mesh. The basis of FEM has the local support which leads to 
the sparsity of the stiff matrices. This property is suitable for 
the modern high performance computers. 

\revise{In this section, we introduce the way to build the NN element basis 
by combing the finite element basis and NN functions. 
The deduced NN element space not only can handle the complex geometry and boundary value 
conditions, but also has the strong ability of expression. The application of NN functions 
makes the NN element space has strong ability of adaptivity for many singular PDEs}.

\subsection{Envelop function from finite element mesh}
The finite element mesh serves as the foundation for approximating the solution of PDEs. 
The solution is typically expressed as a piecewise polynomial function over the elements, 
and the accuracy of the approximation depends on both the quality of the mesh 
and the degree of the polynomial basis functions.  In this subsection, with the finite 
element mesh, we define some type of piecewise polynomials which will be 
used as the envelop functions for the NN element basis. 
As we know, the application of finite element mesh is to handle the complex
geometric domain and boundary value conditions, since the mesh can represent the 
geometry and the corresponding basis can express the boundary value conditions. 
Figure \ref{FEM_Mesh} shows an example of finite element mesh $\mathcal T_h$ 
on the unit square $[0,1]^2$. 
\begin{figure}[htb]
\centering
\includegraphics[width=8cm,height=6cm]{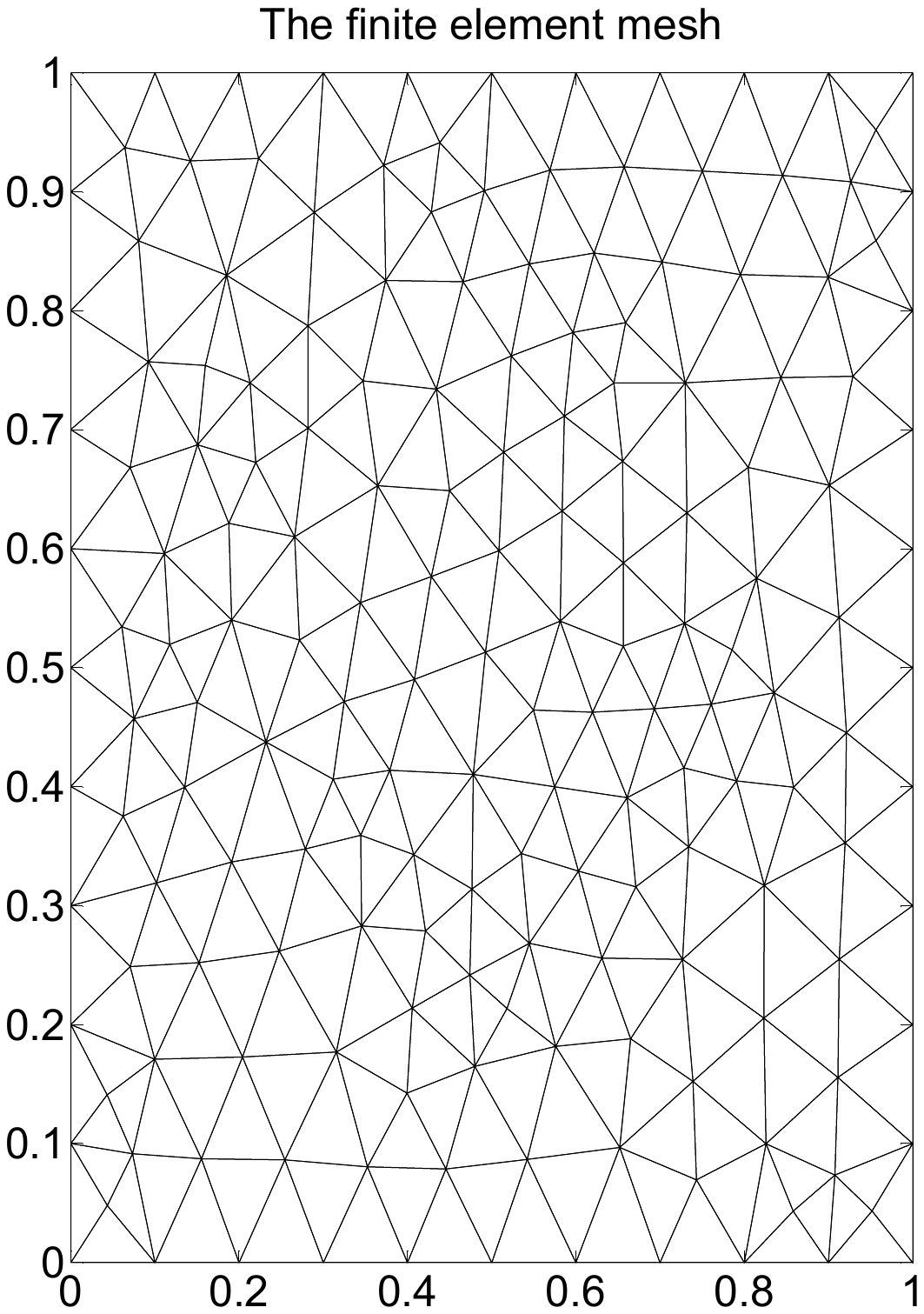}
\vskip-0.4cm 
\caption{The finite element mesh $\mathcal T_h$}\label{FEM_Mesh}
\end{figure}

In this paper, we assume the finite element mesh is regular 
\cite{brenner2007mathematical,ciarlet2002finite}. 
For the following description, \revise{let $\mathcal N_h$ and $\mathcal E_h$ 
denote the set of all vertices and edges of the mesh $\mathcal T_h$}.  
For easy understanding, we are mainly concerned with the triangulation $\mathcal T_h$ 
for the computing domain $\Omega$. Then the barycentric coordinates will be 
used in our description. The following part of this subsection is to 
introduce the local basis as the envelop functions corresponding to the vertex, edge 
and face, respectively.

We come to define the local basis as the envelop function for the vertex $Z\in\mathcal N_h$ 
of the finite element mesh $\mathcal T_h$. For this aim, we build the patch $\omega_Z$ 
associated with the vertex $Z$ by selecting all the elements which include $Z$ 
as one vertex (see Figure \ref{Figure_Node}).  Then the basis is defined as follows
\begin{eqnarray}\label{Envelop_Z}
\varphi_Z(x)|_{K_i} = \lambda_1,\ \ \ {\rm for}\  K_i\in \omega_Z,\ \forall Z\in\mathcal N_h,
\end{eqnarray}
where the vertex $Z$ is set to be the local first vertex in each $K_i\in \omega_Z$.  
\begin{figure}[ht]
\centering
\begin{tikzpicture}
\draw[help lines,color=gray!160,step=10pt,very thick, xshift =100pt,] (6,6) -- (5,8);
\draw[help lines,color=gray!160,step=10pt,very thick, xshift =100pt] (4,6) -- (9,6);
\draw[help lines,color=gray!160,step=10pt,very thick, xshift =100pt] (6,6) -- (8,8);
\draw[help lines,color=gray!160,step=10pt,very thick, xshift =100pt] (6,6) -- (5,4);
\draw[help lines,color=gray!160,step=10pt,very thick, xshift =100pt] (6,6) -- (7,4);
\draw[help lines,color=gray!160,step=10pt,very thick, xshift =100pt] (4,6) -- (5,8);
\draw[help lines,color=gray!160,step=10pt,very thick, xshift =100pt] (5,8) -- (8,8);
\draw[help lines,color=gray!160,step=10pt,very thick, xshift =100pt] (8,8) -- (9,6);
\draw[help lines,color=gray!160,step=10pt,very thick, xshift =100pt] (9,6) -- (7,4);
\draw[help lines,color=gray!160,step=10pt,very thick, xshift =100pt] (5,4) -- (7,4);
\draw[help lines,color=gray!160,step=10pt,very thick, xshift =100pt] (5,4) -- (4,6);
\node at (9.3,6.5) [right] {$Z$}; 
\end{tikzpicture}
\caption{The patch $\omega_Z$ for one node $Z$}\label{Figure_Node}
\end{figure}
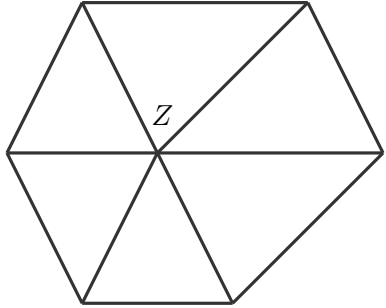

For any interior edge $E\in\mathcal E_h$, the corresponding patch $\omega_E$ includes 
two triangles denoted by $K_1$ and $K_2$, i.e. $\omega_E=K_1\cup K_2$ (see Figure \ref{Figure_Edge}). 
Then the corresponding basis as the envelop function is defined as follows  
\begin{eqnarray}\label{Envelop_E}
\varphi_E(x)|_{K_i} = \lambda_2\lambda_3,\ \ \ {\rm for}\ K_i\in \omega_E,\ \forall E\in \mathcal E_h.
\end{eqnarray}
\begin{figure}[ht]
\centering
\begin{tikzpicture}
\draw[help lines,color=gray!160,step=10pt,very thick, xshift =100pt] (6,6) -- (10,6);
\draw[help lines,color=gray!160,step=10pt,very thick, xshift =100pt] (6,6) -- (8,8);
\draw[help lines,color=gray!160,step=10pt,very thick, xshift =100pt] (6,6) -- (7,4);
\draw[help lines,color=gray!160,step=10pt,very thick, xshift =100pt] (8,8) -- (10,6);
\draw[help lines,color=gray!160,step=10pt,very thick, xshift =100pt] (10,6) -- (7,4);
\node at (11.2,5.8) [right] {$E$}; 
\node at (9,6) [right] {$2$}; 
\node at (13.5,6) [right] {$3$}; 
\node at (11.2,6.8) [right] {$K_1$}; 
\node at (10.5,5) [right] {$K_2$}; 
\node at (11.25,8.3) [right] {$1$}; 
\node at (10.27,3.7) [right] {$1$}; 
\end{tikzpicture}
\caption{The patch $\omega_E$ for one edge $E$ and the corresponding basis 
$\lambda_2\lambda_3$}\label{Figure_Edge}
\end{figure}
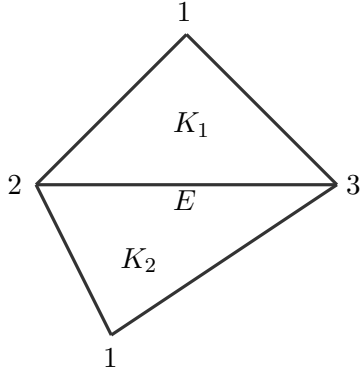

For any triangle $K\in \mathcal T_h$ with the vertices $Z_1$, $Z_2$ and $Z_3$ (see Figure \ref{Figure_Triangle}). 
Then the corresponding basis is defined as follows
\begin{eqnarray}\label{Envelop_K}
\varphi_K(x) = \lambda_1\lambda_2\lambda_3,\ \ \ \forall K\in \mathcal T_h.
\end{eqnarray}
\begin{figure}[ht]
\centering
\begin{tikzpicture}
\draw[help lines,color=gray!160,step=10pt,very thick, xshift =100pt] (6,6) -- (10,6);
\draw[help lines,color=gray!160,step=10pt,very thick, xshift =100pt] (6,6) -- (8.5,8.2);
\draw[help lines,color=gray!160,step=10pt,very thick, xshift =100pt] (8.5,8.2) -- (10,6);
\node at (11.45,6.8) [right] {$K$}; 
\node at (11.75,8.4) [right] {$3$}; 
\node at (9.1,6) [right] {$1$}; 
\node at (13.5,6) [right] {$2$}; 
\end{tikzpicture}
\caption{The triangle $K$ and the corresponding basis 
$\lambda_1\lambda_2\lambda_3$}\label{Figure_Triangle}
\end{figure}
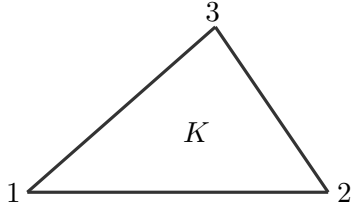

The local basis $\varphi_Z(x)$, $\varphi_E(x)$ and $\varphi_K(x)$ will act as the 
envelop functions multiplied by the NN function to build the NN element basis. 
For simplicity of notation and following description, we use 
$\varphi_1(x)$, $\cdots$, $\varphi_N(x)$ to denote  all the envelop functions on the 
finite element mesh $\mathcal T_h$ according to some type of order of the 
vertices, edges and triangles.  The corresponding patches corresponding to the 
vertices, edges and triangles are denoted by $\Omega_i$ according to the same order of 
$\varphi_1(x)$, $\cdots$, $\varphi_N(x)$ in the mesh $\mathcal T_h$, i.e., 
$\Omega_i:={\rm supp}(\varphi_i(x))$, where ${\rm supp}(f(x))$ denotes the support 
of the function $f(x)$. 
 
%---------------------------------------------------------------------------------------
\subsection{Neural network element basis}\label{Section_TNN_Archictecture}
This subsection is devoted to introducing the NN element trial space based on the 
envelop functions defined on the finite element mesh $\mathcal T_h$ and the local NN functions.  
In order to express clearly and facilitate the construction of the NN element method for 
solving PDEs, we will also introduce some important definitions and properties here.

Based on the finite element basis defined in the last subsection, the 
NN element basis functions corresponding to vertex, edge and triangle can be defined as follows
\begin{eqnarray}\label{Basis_Vertex}
\varphi_Z(x,\theta) = \varphi_Z(x)\cdot \phi_Z(x,\theta),\ \ \ \forall Z\in\mathcal N_h,
\end{eqnarray}
\begin{eqnarray}\label{Basis_Edge}
\varphi_E(x,\theta) = \varphi_E(x)\cdot \phi_E(x,\theta),\ \ \ \forall E\in\mathcal E_h,
\end{eqnarray}
and 
\begin{eqnarray}\label{Basis_K}
\varphi_K(x,\theta) = \varphi_K(x)\cdot \phi_K(x,\theta),\ \ \ \forall K\in\mathcal T_h, 
\end{eqnarray}
where $\varphi_Z(x)$, $\varphi_E(x)$ and $\omega_K(x)$ are defined 
by (\ref{Envelop_Z}), (\ref{Envelop_E}) and (\ref{Envelop_K}), 
$\phi_Z(x,\theta)$, $\phi_Z(x,\theta)$ and $\phi_K(x,\theta)$ denote 
the local NN functions on the patch $\omega_Z$, $\omega_E$ and $K$, respectively. 

Based on the construction way in (\ref{Basis_Vertex}), (\ref{Basis_Edge}) and (\ref{Basis_K}),  
the NN-element space $V_{\rm NNh}(\mathcal T_h,\theta)$ can be defined as follows
\begin{eqnarray}
V_{\rm NNh}(\mathcal T_h,\theta) &=& {\rm span}\Big\{\varphi_Z(x,\theta), \forall Z\in\mathcal N_h \Big\}
\bigcup{\rm span}\Big\{\varphi_E(x,\theta), \forall E\in\mathcal E_h \Big\}\nonumber\\
&&  \bigcup {\rm span}\Big\{\varphi_K(x,\theta), \forall K\in\mathcal T_h \Big\}.
\end{eqnarray}
It is easy to know that $V_{\rm NNh}(\mathcal T_h,\theta) \subset H_0^1(\Omega)$. 
For simplicity of notation, let $N:={\rm dim}(V_{\rm NNh}(\mathcal T_h, \theta))$ and $\varphi_1(x,\theta)$, 
$\cdots$, $\varphi_N(x,\theta)$ denote the basis of the space $V_{\rm NNh}(\mathcal T_h, \theta)$ 
according to the same order of the envelop functions $\varphi_1(x)$, $\cdots$, $\varphi_N(x)$. 
Then the function $\psi(x,c,\theta)\in V_{\rm NNh}$ can be expressed as the following 
linear combination with the coefficient vector $c\in \mathbb R^N$ 
\begin{eqnarray}
\psi(x,c,\theta) = \sum_{Z\in\mathcal N_h}c_Z\varphi_Z(x,\theta)+\sum_{E\in\mathcal E_h}c_E\varphi_E(x,\theta)
+\sum_{K\in\mathcal T_h}c_K\varphi_K(x,\theta)=\sum_{i=1}^Nc_i\varphi_i(x,\theta),
\end{eqnarray}
where $c$ denotes the vector with the elements $c_i$, $i=1, \cdots, N$.

In order to show the way here more clearly, we take the L shape domain and its mesh 
as an example. Figure \ref{Figure_4_Subdomains} shows the corresponding mesh and the corresponding 
vertices, edges and triangles for building the space $V_{\rm NNh}(\mathcal T_h,\theta)$.  
Here, we assume the problem has the homogeneous Dirichlet boundary condition. 
\begin{figure}[ht]
\centering
\begin{tikzpicture}
\node at (11.45,8.25) [right] {$Z_1$}; 
\node at (11.2,12.25) [right] {$Z_2$}; 
\node at (7.1,12.25)  [right] {$Z_3$}; 
\node at (6.8,8)      [right] {$Z_4$}; 
\node at (7.2,3.7)    [right] {$Z_5$}; 
\node at (11.25,3.7)  [right] {$Z_6$}; 
\node at (15.2,3.7)   [right] {$Z_7$}; 
\node at (15.2,8.25)  [right] {$Z_8$}; 
\node at (9.25,10.5)  [right] {$E_1$}; 
\node at (8.25,11)    [right] {$K_1$}; 
\node at (9.35,7.75)  [right] {$E_2$}; 
\node at (9.55,6)     [right] {$E_3$};
\node at (9.95,9.25)  [right] {$K_2$};  
\node at (9.95,6.9)   [right] {$K_3$}; 
\node at (12.5,6.9)   [right] {$K_5$};
\node at (11.4,6.2)   [right] {$E_4$}; 
\node at (8.25,5.25)  [right] {$K_4$}; 
\node at (13.6,6.05)  [right] {$E_5$}; 
\node at (14,5) [right] {$K_6$};
\draw[help lines,color=gray!160,step=10pt,very thick, xshift =100pt] (4,4)  -- (12,4);
\draw[help lines,color=gray!160,step=10pt,very thick, xshift =100pt] (12,4) -- (12,8);
\draw[help lines,color=gray!160,step=10pt,very thick, xshift =100pt] (8,12) -- (4,12);
\draw[help lines,color=gray!160,step=10pt,very thick, xshift =100pt] (4,12) -- (4,4);
\draw[help lines,color=gray!160,step=10pt,very thick, xshift =100pt] (8,4)  -- (8,8);
\draw[help lines,color=gray!160,step=10pt,very thick, xshift =100pt] (4,8)  -- (8,8);
\draw[help lines,color=gray!160,step=10pt,very thick, xshift =100pt] (8,8)  -- (8,12);
\draw[help lines,color=gray!160,step=10pt,very thick, xshift =100pt] (8,8)  -- (12,8);
\draw[help lines,color=gray!160,step=10pt,very thick, xshift =100pt] (4,8)  -- (8,12);
\draw[help lines,color=gray!160,step=10pt,very thick, xshift =100pt] (4,8)  -- (8,4);
\draw[help lines,color=gray!160,step=10pt,very thick, xshift =100pt] (12,8) -- (8,4);
\end{tikzpicture}
\caption{The mesh and corresponding vertices, edges and triangles for the 
L shape domain}\label{Figure_4_Subdomains}
\end{figure}
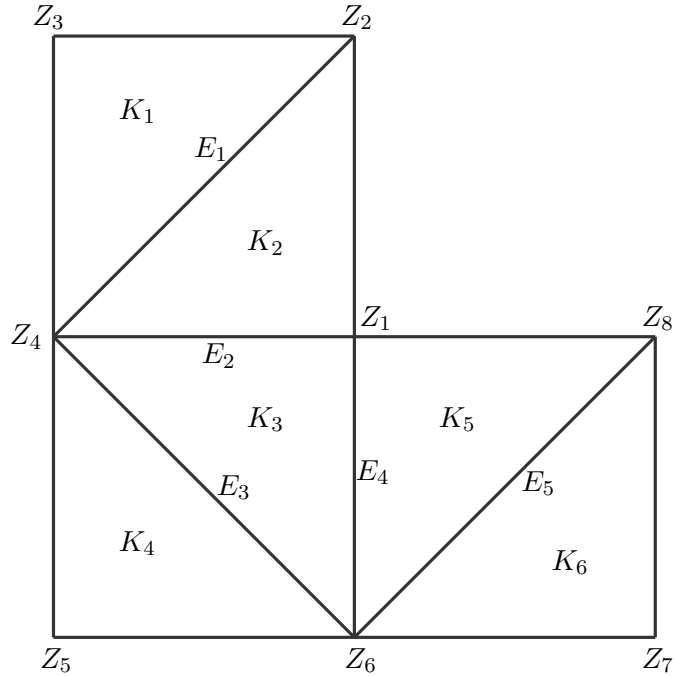
Since the vertices are all on the Dirichlet boundary, there is no free vertices and then 
$\mathcal N_h=\varnothing$. According to the Dirichlet boundary, it is easy to know 
that $\mathcal E_h = \{E_1, E_2, E_3, E_4,  E_5\}$ 
and $\mathcal T_h = \{K_1, K_2, K_3, K_4,  K_5, K_6\}$. 
Then we can build the NN element basis functions with the methods (\ref{Basis_Edge}) 
and (\ref{Basis_K}) and the NN element space is defined as follows
\begin{eqnarray*}
V_{\rm NNh}(\mathcal T_h,\theta) 
= {\rm span}\Big\{\varphi_{E_1}(x,\theta), \cdots, \varphi_{E_5}(x,\theta)\Big\} 
\bigcup {\rm span}\Big\{\varphi_{K_1}(x,\theta), \cdots, \varphi_{K_6}(x,\theta)\Big\}. 
\end{eqnarray*}
Of course, the space part ${\rm span}\Big\{\varphi_{K_1}(x,\theta), \cdots, \varphi_{K_6}(x,\theta)\Big\}$ 
according to the elements is not necessary. 

\subsection{Error estimates of NN element approximation}
In this subsection, we use the idea of partition of unity \cite{MelenkBabuska} to give 
the error analysis for the NN element approximation. For this aim, we define the 
following partition functions based on the envelop functions 
defined in (\ref{Basis_Vertex}), (\ref{Basis_Edge}) and (\ref{Basis_K}) on the mesh $\mathcal T_h$ 
\begin{eqnarray}\label{Partition_Function}
\psi_i(x) = \frac{\varphi_i(x)}{\sum_{j=1}^N\varphi_j(x)},\ \ \ \ i=1, \cdots, N.
\end{eqnarray}
Then it is easy to know 
\begin{eqnarray}
\sum_{i=1}^N \psi_i(x) = 1,
\end{eqnarray}
and the support of $\psi_i(x)$ is the same as $\varphi_i(x)$, $i=1, \cdots, N$.

For easy understanding, we take a triangle $K$ and its patch $\omega_K = \{K'| K\cap K' \neq \varnothing\}$ 
in Figure \ref{Figure_Patch} as an example 
to show the construction of the partition functions. 
According to the notation in Figure \ref{Figure_Patch}, 
the envelop functions on the triangle $K$ can be denoted and ordered as follows
\begin{eqnarray*}
&&\varphi_1(x) = \varphi_{Z_1}(x), \ \ \varphi_2(x) = \varphi_{Z_2}(x), \ \  \varphi_3(x) = \varphi_{Z_3}(x),   \ \ 
\varphi_4(x) = \varphi_{E_1}(x),\\
&&\varphi_5(x) = \varphi_{E_2}(x),  \ \ \varphi_6(x) = \varphi_{E_3}(x),  \ \ 
\varphi_7(x) = \varphi_{K}(x).
\end{eqnarray*}
\begin{figure}[ht]
\centering
\begin{tikzpicture}
\node at (9.7,7.2)   [right] {$K$};
\node at (11.2,8.5) [right] {$Z_1$};  
\node at (6.8,8.25)   [right] {$Z_2$}; 
\node at (11.45,3.75) [right] {$Z_3$};  
\node at (9.35,8.4)  [right] {$E_3$};  
\node at (8.7,6.5)     [right] {$E_1$};  
\node at (11.25,6.8)   [right] {$E_2$}; 
\draw[help lines,color=gray!160,step=10pt,very thick, xshift =100pt] (4.1,4.3)  -- (8,5);
\draw[help lines,color=gray!160,step=10pt,very thick, xshift =100pt] (10.5,4) -- (11,8);
\draw[help lines,color=gray!160,step=10pt,very thick, xshift =100pt] (8,11) -- (4.5,11);
\draw[help lines,color=gray!160,step=10pt,very thick, xshift =100pt] (8,5)  -- (10.5,4);
\draw[help lines,color=gray!160,step=10pt,very thick, xshift =100pt] (4,8)  -- (4.1,4.3);
\draw[help lines,color=gray!160,step=10pt,very thick, xshift =100pt] (8,5)  -- (7.7,8.3);
\draw[help lines,color=gray!160,step=10pt,very thick, xshift =100pt] (4,8)  -- (7.7,8.3);
\draw[help lines,color=gray!160,step=10pt,very thick, xshift =100pt] (7.7,8.3)  -- (8,11);
\draw[help lines,color=gray!160,step=10pt,very thick, xshift =100pt] (7.7,8.3)  -- (11,8);
\draw[help lines,color=gray!160,step=10pt,very thick, xshift =100pt] (4,8)  -- (8,11);
\draw[help lines,color=gray!160,step=10pt,very thick, xshift =100pt] (4,8)  -- (8,5);
\draw[help lines,color=gray!160,step=10pt,very thick, xshift =100pt] (11,8) -- (8,5);
\draw[help lines,color=gray!160,step=10pt,very thick, xshift =100pt] (11,8) -- (8,11);
\draw[help lines,color=gray!160,step=10pt,very thick, xshift =100pt] (4,8)  -- (4.5,11);
\draw[help lines,color=gray!160,step=10pt,very thick, xshift =100pt] (1.5,8)  -- (4,8);
\draw[help lines,color=gray!160,step=10pt,very thick, xshift =100pt] (1.5,8)  -- (4.1,4.3);
\draw[help lines,color=gray!160,step=10pt,very thick, xshift =100pt] (1.5,8)  -- (4.5,11);
\draw[help lines,color=gray!160,step=10pt,very thick, xshift =100pt] (8,3)  -- (8,5);
\draw[help lines,color=gray!160,step=10pt,very thick, xshift =100pt] (8,3)  -- (4.1,4.3);
\draw[help lines,color=gray!160,step=10pt,very thick, xshift =100pt] (8,3)  -- (10.5,4);
\end{tikzpicture}
\caption{The patch $\omega_K = \{K'| K\cap K' \neq \varnothing\}$ 
for the triangle $K$.}\label{Figure_Patch}
\end{figure}
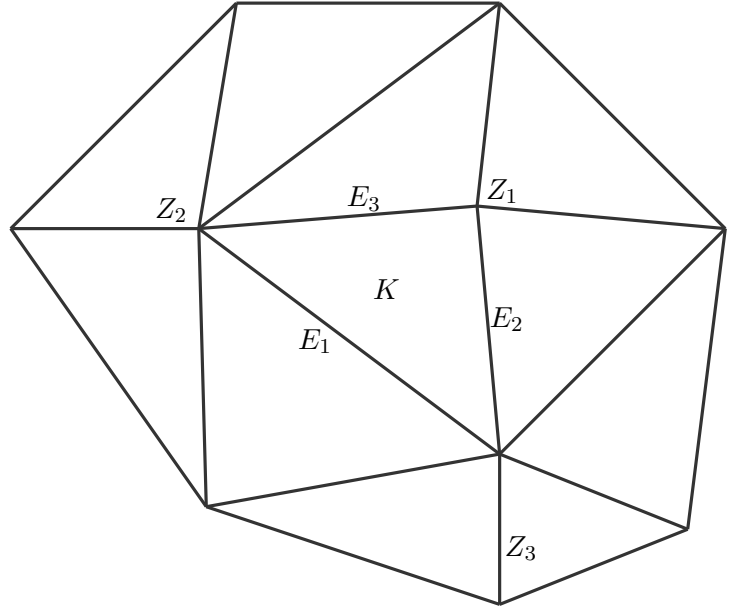
%\begin{figure}[ht]
%\centering
%\begin{tikzpicture} 
%\node at (7.2,3.7)    [right] {$Z_1$}; 
%\node at (10.25,4.2)  [right] {$E_3$};  
%\node at (10.25,7.75)  [right] {$Z_3$};  
%\node at (10.3,5.7)  [right] {$K$}; 
%\node at (12,6.4)   [right] {$E_1$}; 
%\node at (8.5,6.25)  [right] {$E_2$}; 
%\node at (13.5,5) [right] {$Z_2$};
%\draw[help lines,color=gray!160,step=10pt,very thick, xshift =100pt] (4,4)  -- (10,5);
%\draw[help lines,color=gray!160,step=10pt,very thick, xshift =100pt] (10,5) -- (7,7.5);
%\draw[help lines,color=gray!160,step=10pt,very thick, xshift =100pt] (7,7.5) -- (4,4);
%\end{tikzpicture}
%\caption{The element $K$ and its vertices, edges}\label{Figure_Element}
%\end{figure}
Based on the local support $\varphi_i(x)$ and $\psi_i(x)$, the associated partition functions 
restricted on $K$ can be defined as follows 
\begin{eqnarray}
\psi_i(x)|_K = \frac{\varphi_i(x)|_K}{\sum_{j=1}^7 \varphi_j(x)},\ \ \ i=1, \cdots, 7.
\end{eqnarray}
It is easy to know 
\begin{eqnarray}
\left(\sum_{i=1}^N \psi_i(x)\right)\Big|_K = 1 
\end{eqnarray}
and the support for the partition functions $\psi_1(x)$, $\cdots$, $\psi_7(x)$ 
is the patch as in Figure \ref{Figure_Patch}.

%In conclusion, the approximation properties of NN element method are based on the facts that
%\begin{enumerate}
%\item (local approximability) a smooth function can be approximated locally by the neural network;
%
%\item (conformity of the NN element spaces/interelement continuity) NN spaces
%are big enough to absorb extra constraints of continuity across interelement boundaries 
%without loosing the approximation properties.
%\end{enumerate}
%Conversely, any system of functions which have good local approximation properties and
%can be constrained to satisfy some interelement continuity leads to a good finite element
%method. 

%Let us first elaborate the problem of local approximability. There are many works 
%to show that the NN functions have good local approximation properties. 
%%For certain types of equations,  
%%one can exploit the structure of the differential equation to construct spaces of functions
%%which can approximate the solution even better than the spaces of polynomials. 

%Let us now turn to the problem of conformity of the NN element space/interelement 
%continuity. It is well known that the NN has good local approximation properties 
%for the approximation of a solution $u$ of a differential equation. 
%In general, it is very difficult to enforce conformity, i.e., interelement continuity, 
%for the local NN approximation spaces. The PUNNEM, however, 
%offers a means to construct a conforming space out of any given system of local 
%approximation spaces without sacrificing the approximation properties. 
%This is done as follows. 

\revise{
The overlapping patches $\left\{\Omega_i\right\}$ cover the domain $\Omega$ and 
$\left\{\psi_i(x)\right\}$ be a partition of unity subordinate to the cover. 
On each patch $\Omega_i$,
let \comm{$V_i \coloneqq \{\phi_i(x;\theta_i)\ |\ \theta_i\in\Theta_i\} \subset H^1\left(\Omega_i\right)$}
denotes the local NN trial space. %by which $u|_{\Omega_i}$ can be approximated well. 
The global NN element space $V$ is then defined by $\sum_{i=1}^N \psi_i V_i$. 
For any function $u$, it is well known that $u|_{\Omega_i}$ can be approximated very well by the 
local NN trial space $V_i$. }

%enforces the conformity of the global space $V$.

%In this section, we present a method of constructing conforming subspaces of $H^1(\Omega)$. 
%We construct NN element spaces which are subspaces of $H^1(\Omega)$ 
%as an example because of their importance in applications. We would like 
%to stress at this point that the method leads to the construction of smoother 
%spaces (subspaces of $H^k, k>1$ ) or subspaces of Sobolev spaces $W^{k, p}$ 
%in a straightforward manner. 

%The main technical notion in the construction of 
%the PUNNEM spaces is the ($M, C_{\infty}, C_G$) partition of unity.
\begin{lemma}\label{Definition_1}
%Let $\Omega \subset \mathbb{R}^{n}$ be an open set, 
The cover $\{\Omega_i\}_{i=1}^N$ of $\Omega$ satisfies the following pointwise overlap condition
\begin{eqnarray}
\exists M \in \mathbb{N}, \quad \forall x \in \Omega, \quad 
\operatorname{card}\left\{K\in \mathcal T_h \mid x \in K\right\} \leqslant M. \label{Regular_Estimate}
\end{eqnarray}
Then $\left\{\psi_i\right\}_{i=1}^N$ is a Lipschitz partition of 
unity subordinate to the cover $\mathcal T_h$ satisfying
\begin{eqnarray}
&&\operatorname{supp} (\psi_i) \subseteq \Omega_i, \quad 1\leq i\leq N, \\
&&\sum_{i=1}^N \psi_i(x) \equiv 1 \quad \text {\rm on } \Omega, \\
&&\left\|\psi_i\right\|_{L^{\infty}\left(\mathbb R^n\right)} \leqslant C_{\infty},\quad i=1, \cdots, N,\\
&&\left\|\nabla \psi_i\right\|_{L^2\left(\mathbb{R}^n\right)} 
\leqslant \frac{C_G}{{\rm diam}(\Omega_i)}, \quad i=1, \cdots, N,
\end{eqnarray}
where ${\rm diam}(\Omega_i)$ denotes the diameter of $\Omega_i$, $C_{\infty}$ and $C_G$ are two constants 
independent of the mesh size. 
%Then $\left\{\psi_i\right\}$ is called a ($M, C_{\infty}, C_G$) partition of unity subordinate 
%to the cover $\left\{\Omega_i\right\}$. 
%The partition of unity $\left\{\psi_i\right\}$ is said to be of degree $m \in \mathbb{N}_0$ 
%if $\left\{\psi_i\right\} \subset C^m\left(\mathbb{R}^n\right)$. 
%The covering sets $\left\{\Omega_i\right\}$ are called patches.
\end{lemma}
\begin{proof}
The proof can be given by using the regularity of the mesh $\mathcal T_h$ and the property of the 
envelop functions defined in this paper \cite{brenner2007mathematical,ciarlet2002finite,MelenkBabuska}. 
\end{proof}

%\begin{definition}\label{Definition_2}
%Let $\left\{\Omega_i\right\}$ be an open cover of $\Omega \subset \mathbb{R}^n$ 
%and let $\left\{\psi_i\right\}$ be a $\left(M, C_{\infty}, C_G\right)$ partition of unity subordinate 
%to $\left\{\Omega_i\right\}$. Let $V_i \subset H^1\left(\Omega_i \cap \Omega\right)$ be the NN trial space (set). 
%Then the space
%$$
%V_{\rm NNh}:=\sum_{i=1}^N \psi_i \left(c_i\phi_i(x,\theta)\right)
%=\left\{\sum_{i=1}^N \psi_i v_i \Big| v_i \in V_i\right\} \subset H^1(\Omega)
%$$
%is called the partition of unity of NN element space (PUNNEM). 
%The PUNNEM space $V$ is said to be of degree $m \in \mathbb{N}$ if $V \subset C^m(\Omega)$. 
%The spaces $V_i$ are referred to as the local NN element approximation spaces.
%\end{definition}

Then following the idea from \cite[Theorem 2.1]{MelenkBabuska}, we can give the error estimates for 
the NN element approximation. 
\begin{theorem}\label{Theorem_Error_Estimates}
%Let $\Omega \subset \mathbb{R}^n$ be given. 
%Let $\left\{\Omega_i\right\},\left\{\psi_i\right\}$ and $\left\{V_i\right\}$ 
%be as in Definitions \ref{Definition_1} and \ref{Definition_2}. 
Let $u \in H^1(\Omega)$ be the function to be approximated. 
Assume that the local approximation spaces $V_i$ have 
the following approximation properties: 
On each patch $\Omega_i \cap \Omega$, $u$ can be approximated by a \comm{NN $v_i \in V_i$} such that
\begin{eqnarray}
&& \left\|u-v_i\right\|_{L^2(\Omega_i \cap \Omega)} \leqslant \varepsilon_1(i), \label{Error_L2_Local}\\
&& \left\|\nabla\left(u-v_i\right)\right\|_{L^2(\Omega_i\cap \Omega)} 
\leqslant \varepsilon_2(i).\label{Error_H1_Local}
\end{eqnarray}
Then the NN element \comm{approximation}
$$
u_{\rm NNh}=\sum_{i=1}^N \psi_i v_i \in V \subset H^1(\Omega)
$$
has following error estimates 
\begin{eqnarray}
&& \left\|u-u_{\rm NNh}\right\|_{L^2(\Omega)} \leqslant \sqrt{M} C_{\infty}
\left(\sum_{i=1}^N \varepsilon_1^2(i)\right)^{1/2},\label{Error_Estimate_1} \\
&& \left\|\nabla\left(u-u_{\rm NNh}\right)\right\|_{L^2(\Omega)} 
\leqslant \sqrt{2 M}\left(\sum_{i=1}^N\left(\frac{C_G}{\operatorname{diam}(\Omega_i)}\right)^2 
\varepsilon_1^2(i)+C_{\infty}^2 \varepsilon_2^2(i)\right)^{1/2}.\label{Error_Estimate_2}
\end{eqnarray}
\end{theorem}
\begin{proof}
We will only show estimate (\ref{Error_Estimate_2}) because (\ref{Error_Estimate_1}) is proved similarly. 
Let $u_{\rm NNh}$ be defined as in the statement of the theorem. 
Since the functions $\psi_i$ form a partition of unity, 
we have $1 \cdot u=\left(\sum_{i=1}^N \psi_i\right) u=$ $\sum_{i=1}^N \psi_i u$ and thus
\begin{eqnarray}\label{Inequality_1}
\left\|\nabla\left(u-u_{\rm NNh}\right)\right\|_{L^2(\Omega)}^2 
&=&\left\|\nabla \sum_{i=1}^N \psi_i\left(u-v_i\right)\right\|_{L^2(\Omega)}^2 \nonumber\\
&\leqslant& 2\left\|\sum_{i=1}^N \nabla \psi_i\left(u-v_i\right)\right\|_{L^2(\Omega)}^2
+2\left\|\sum_{i=1}^N \psi_i \nabla\left(u-v_i\right)\right\|_{L^2(\Omega)}^2.
\end{eqnarray}

Since there are not more than $M$ patches overlap in any given point $x \in \Omega$, 
the sums $\sum_{i=1}^N \nabla \psi_i\left(u-v_i\right)$ and $\sum_{i=1}^N \psi_i \nabla\left(u-v_i\right)$ 
also contain at most $M$ terms for any fixed $x \in \Omega$. 
Thus, the following inequalities 
\begin{eqnarray}\label{Inequality_2}
\left|\sum_{i=1}^N \nabla \psi_i\left(u-v_i\right)\right|^2 
\leqslant M \sum_{i=1}^N \left|\nabla \psi_i\left(u-v_i\right)\right|^2,
\end{eqnarray}
and
\begin{eqnarray}\label{Inequality_3}
\left|\sum_{i=1}^N \psi_i \nabla\left(u-v_i\right)\right|^2 
\leqslant M \sum_{i=1}^N\left|\psi_i \nabla\left(u-v_i\right)\right|^2,
\end{eqnarray}
holds  for any $x \in \Omega$. 
  
With the fact $\operatorname{supp}(\psi_i) \subseteq \Omega_i$, 
the following estimates hold
\begin{eqnarray}\label{Inequality_4}
&& 2\left\|\sum_{i=1}^N \nabla \psi_i\left(u-u_i\right)\right\|_{L^2(\Omega)}^2
+2\left\|\sum_{i=1}^N \psi_i \nabla\left(u-v_i\right)\right\|_{L^2(\Omega)}^2 \\
&& \quad \leqslant 2 M \sum_{i=1}^N\left\|\nabla \psi_i\left(u-v_i\right)\right\|_{L^2(\Omega)}^2
+2 M \sum_{i=1}^N\left\|\psi_i \nabla\left(u-v_i\right)\right\|_{L^2(\Omega)}^2 \\
&& \quad \leqslant 2 M \sum_{i=1}^N\left\|\nabla \psi_i
\left(u-v_i\right)\right\|_{L^2\left(\Omega_i \cap \Omega\right)}^2
+2 M \sum_{i=1}^N\left\|\psi_i \nabla\left(u-v_i\right)\right\|_{L^2\left(\Omega_i \cap \Omega\right)}^2\\
&& \quad \leqslant 2 M \sum_{i=1}^N
\left(\frac{C_G^2}{\left(\operatorname{diam}(\Omega_i)\right)^2} \varepsilon_1^2(i)
+C_x^2 \varepsilon_2^2(i)\right). 
\end{eqnarray}
Then the desired result (\ref{Error_Estimate_2}) can be deduced by combining (\ref{Inequality_1}), 
(\ref{Inequality_2}), (\ref{Inequality_3}) and (\ref{Inequality_4}) and the proof is complete.  
\end{proof}

It is well known that the NN functions have good local approximation properties which means 
$\varepsilon_1(i)$ and $\varepsilon_2(i)$ can be very small even for small scale NN systems.
Theorem \ref{Theorem_Error_Estimates} shows that the global space $V$ inherits the approximation 
properties of the local spaces $V_i$, i.e., the function $u$ can be approximated 
on $\Omega$ by functions of $V$ as well as the functions $u|_{\Omega_i}$ 
can be approximated in the local spaces $V_i$. The conformity 
or interelement continuity of the NN element space inherit from the application of 
the envelop functions defined on the finite element mesh, which enforces the interelement 
continuity to construct a conforming space out of the local NN spaces 
without sacrificing the approximation properties.

Although the finite-dimensional subspace $V_{\rm NNh}(\mathcal T_h,\theta)$ here 
is constructed based on a finite element mesh, 
the use of neural networks (NNs) allows the accuracy of the approximate solution to 
reach the level of NN approximation without depending on the mesh size. 
The primary reason for using a finite element mesh is to handle complex 
geometric domains, while the neural network is the fundamental factor for improving 
the approximation accuracy. From this perspective, the NN element method effectively 
combines the finite element mesh's ability for handling complex geometries 
and the strong approximation capability of neural networks. 
This synergy enables neural network-based machine learning algorithms 
to solve a broader range of problems arising from engineering applications. 
This is the primary reason and objective behind designing 
the algorithm presented in this paper.

\comm{
\begin{remark}
Theorem \ref{Theorem_Error_Estimates} obtains the upper bound of the NN element approximation errors from the local approximation property of NN.
It is worth mentioning that if the Lipschitz partition of unity $\{\psi_i\}_{i=1}^N$ is the Lagrange finite element basis function of order $k$, the following error estimates based on the finite element method also holds
\begin{eqnarray}
&&\left\|u-u_{\rm NNh}\right\|_{L^2(\Omega)} \leq Ch^k,\\
&&\left\|\nabla\left(u-u_{\rm NNh}\right)\right\|_{L^2(\Omega)} \leq Ch^{k+1}.
\end{eqnarray}
This is because common NNs, such as FNN, can represent constant functions.
However, for the numerical stability of the NN training process, we recommend adding constant functions to each local NN trial space, that is $V_i \coloneqq\{1\}\cup \{\phi_i(x;\theta_i)\ |\ \theta_i\in\Theta_i\}$.
\end{remark}
}

\section{NN element based machine learning method}\label{Section_Machine_Learning}
In this section, we introduce the machine learning method by using the NN element 
space $V_{\rm NNh}(\mathcal T_h,\theta)$ and optimization process for some type of loss functions. 
For simplicity of notation, we also use $V_{\rm NNh}$ to denote $V_{\rm NNh}(\mathcal T_h,\theta)$ here.

The same as to the finite element method, based on the NN element space $V_{\rm NNh}$, we can define the 
corresponding Galerkin scheme as follows: Find $u_h\in V_{\rm NNh}$ such that 
\begin{eqnarray}\label{NN_Element_Problem}
a(u_h, v_h) = (f,v_h),\ \ \ \forall v_h\in V_{\rm NNh},
\end{eqnarray}
where 
\begin{eqnarray}\label{Bilinear_Form}
a(w,v) = \int_\Omega\left( \mathcal A\nabla w\cdot\nabla v + bwv\right)d\Omega,\ \ \ 
(f,v)=\int_\Omega fvd\Omega,\ \ \forall w, v\in H_0^1(\Omega). 
\end{eqnarray}

Based on the subspace approximation theory from the finite element method \cite{brenner2007mathematical, ciarlet2002finite}, the following optimal approximation property holds 
\begin{eqnarray}\label{Optimal_Approximation}
\|u-u_h\|_a = \inf_{v_h\in V_{\rm NNh}(\mathcal T_h,\theta)}\|u-v_h\|_a. 
\end{eqnarray}
%In the following part of this section, we use $\Psi(x,\theta)$

From (\ref{Optimal_Approximation}), the strong approximation ability of NN can improve the accuracy 
and approximation efficiency for the NN element method. 
Since there is a background mesh $\mathcal T_h$, we can also use the techniques 
from the finite element method to assemble the 
stiffness matrix and right hand side term of the equation (\ref{NN_Element_Problem}).  
Given the basis functions of the NN element space, 
we can construct the algebraic form of the discrete equation, namely, 
assemble the stiffness matrix and the right-hand side. 
Since there is also a finite element mesh involved, we can adopt the 
element-wise assembly procedure of the finite element method 
to build the stiffness matrix and the right-hand side.

When assembling the stiffness matrix and the right-hand side, we assume that the NN element space is fixed. 
After obtaining the NN element solution, we can update the entire NN element space 
with the adaptivity steps for a specific loss function to further 
improve the approximating accuracy. % of the NN element solution. 
We refer to this process as the training step of the machine learning 
algorithm for the NN element. 

Fortunately, considering the training process of machine learning method, we can find that this training process 
is naturally an adaptive process for the NN element space. After the $\ell$-th training step, 
the corresponding NN element space is 
\begin{eqnarray}\label{def_Vpl}
V_{\rm NNh}(\mathcal T_h, \theta^{(\ell)}):=\operatorname{span}\left\{\varphi_{j}(x ; \theta^{(\ell)}), 
\ \ \ j=1, \cdots, N\right\}.
\end{eqnarray}
Then, the parameters of the $\ell+1$-th step are updated according 
to the optimization step with a loss function $\mathcal{L}$. 
If the gradient descent method is used, this update can be expressed as follows
\begin{eqnarray}\label{Optimization_Step}
\theta^{(\ell+1)} \leftarrow \theta^{(\ell)} - 
\gamma \frac{\partial\mathcal{L}}{\partial \theta},
\end{eqnarray}
where $\gamma$ denotes the learning rate.
Note that updating parameter $\theta^{(\ell)}$ essentially updates the subspace 
$V_{\rm NNh}(\mathcal T_h, \theta^{(\ell+1)})$. 
In other words, in the training process, an optimal $N$-dimensional subspace 
$V_{\rm NNh}(\mathcal T_h, \theta^{(\ell+1)})$ 
can be selected adaptively according to the loss function. The definition of the loss function determines 
how the subspace is updated and whether the algorithm can ultimately achieve high precision.

After updating the NN element subspace, we can continue solving the equation \eqref{Optimal_Approximation} 
in the new NN element space and then do the iteration until stoping. 
The corresponding numerical method can be defined in Algorithm \ref{Algorithm_NNElement}, 
where the output $\Psi(x;c^{(M)}, \theta^{(M)})$ is the NN element approximation to the 
equation (\ref{Elliptic_Problem}). 
\begin{algorithm}[htb!]
\caption{NN element-based method for homogeneous boundary value problem}\label{Algorithm_NNElement}
\begin{enumerate}
\item Initialization step: Generate the finite element mesh $\mathcal T_h$ 
and build the NN element basis and corresponding NN element space $V_{\rm NNh}(\mathcal T_h, \theta^{(0)})$. 
Set the maximum training steps $M$, learning rate $\gamma$ and $\ell=0$.

\item %Define the $p$-dimensional space \revise{$V_{0,p}^{(\ell)}$} as follows
%\begin{eqnarray*}
%\revise{V_{0,p}^{(\ell)}}:={\rm span}\left\{\varphi_{j}(x;\theta^{(\ell)})
%:=\prod_{i=1}^{d}\widehat\phi_{i,j}
%\left(x_{i};\theta_{i}^{(\ell)}\right)(x_i-a_i)(b_i-x_i), j=1, \cdots, p\right\}.
%\end{eqnarray*}
Assemble the stiffness matrix $A^{(\ell)}$ and the right-hand side term 
$B^{(\ell)}$ on $V_{\rm NNh}(\mathcal T_h,\theta^{(\ell)})$. The entries are defined as follows
\begin{eqnarray*}
&&A_{m,n}^{(\ell)}=a(\varphi_n^{(\ell)},\varphi_m^{(\ell)})
=(\mathcal A\nabla\varphi_n^{(\ell)},\nabla\varphi_m^{(\ell)})+(b\varphi_n^{(\ell)},\varphi_m^{(\ell)}), 
\ 1\leq m,n\leq N, \\
&&B_m^{(\ell)}=(f,\varphi_m^{(\ell)}),\ 1\leq m\leq N.
\end{eqnarray*}

\item Solve the following linear 
 equation to obtain the solution $c\in\mathbb R^{N\times 1}$
\begin{eqnarray*}
A^{(\ell)}c=B^{(\ell)}.
\end{eqnarray*}
Update the coefficient parameter as $c^{(\ell+1)}=c$.
Then the Galerkin approximation on the space $V_{\rm NNh}(\mathcal T_h,\theta^{(\ell)})$ for problem 
(\ref{Elliptic_Problem}) is $\Psi(x;c^{(\ell+1)},\theta^{(\ell)})$.

\item Compute the loss function $\mathcal L^{(\ell+1)}(c^{(\ell+1)},\theta^{(\ell)})$ for the current 
approximation $\Psi(x;c^{(\ell+1)},\theta^{(\ell)})$. 
%\begin{eqnarray*}
%\mathcal L^{(\ell+1)}(c^{(\ell+1)},\theta^{(\ell)}):=
%\eta(\Psi(x;c^{(\ell+1)},\theta^{(\ell)}),\nabla\Psi(x;c^{(\ell+1)},\theta^{(\ell)})).
%\end{eqnarray*}

\item Update the neural network parameter $\theta^{(\ell)}$  as follows 
\begin{eqnarray*}
\theta^{(\ell+1)}=\theta^{(\ell)}-\gamma\frac{\partial\mathcal L^{(\ell+1)}}{\partial\theta}
(c^{(\ell+1)},\theta^{(\ell)}).
\end{eqnarray*}
Then the NN element space is updated to $V_{\rm NNh}(\mathcal T_h,\theta^{(\ell+1)})$. 

\item Set $\ell=\ell+1$ and go to Step 2 for the next step until $\ell=M$.

\end{enumerate}
\end{algorithm}

The NN element method, defined by Algorithm \ref{Algorithm_NNElement}, combines 
the strong feasibility of finite element method and strong approximation ability of 
the NN functions. In Algorithm \ref{Algorithm_NNElement},  Steps 2-3 are designed 
based on the idea of finite dimensional approximation which comes from the finite element method, 
Steps 4-5 implement the training process to optimize the finite dimensional NN element space 
by updating the NN parameters $\theta^{(\ell)}$.  
Based on the fixed number of parameters, the training steps improve the quality of the 
finite dimensional NN element space without moving the mesh $\mathcal T_h$. 
This means NN element based machine learning method has the same effect as the 
moving mesh methods \cite{HuangRussell,LiTangZhang} 
but without the technique difficulties of mesh moving.

\comm{
\begin{remark}
Like the classical finite element method, NNEM can satisfy the homogeneous Dirichlet boundary conditions by modifying the stiffness matrix and the right-hand side term calculated in step 2 of Algorithm \ref{Algorithm_NNElement}.
More specifically, the stiffness matrix is modified by setting the main 
diagonal entries to 1 and all other entries to 0 in rows and columns 
corresponding to the boundary basis functions.
And set the corresponding rows in the right-hand side term to be 0.

For non-homogeneous Dirichlet boundary conditions $u|_{\partial\Omega}=g$, we divide all NN element basis $\{\psi_i\}$ into internal NN element basis $\{\psi_i^{{\rm in}}\}$ and boundary NN element basis $\{\psi_i^{{\rm bd}}\}$.
First, solve the following NN element approximation on the boundary mesh degenerated by $\mathcal T_h$: Find $c^{\rm bd}=(c^{\rm bd}_1,\cdots,c^{\rm bd}_{N_{\rm bd}})^\top$, satisfying
\begin{eqnarray*}
Dc^{\rm bd}=G,
\end{eqnarray*}
where
\begin{eqnarray*}
D_{i,j}=(\psi_j^{\rm bd},\psi_i^{\rm bd})_{\partial\Omega},\ \ \ G_i=(g,\psi_i^{\rm bd})_{\partial\Omega},\ \ \ i,j=1,\cdots,N_{\rm bd}.
\end{eqnarray*}
Then the internal approximation of the original problem can be obtained by solving the following problem: Find $c^{\rm in}=(c^{\rm in}_1,\cdots,c^{\rm in}_{N_{\rm in}})^\top$, satisfying
\begin{eqnarray*}
Ac^{\rm in}=f-b^Tc^{\rm bd},
\end{eqnarray*}
where
\begin{eqnarray}
&&A_{ij}=a(\psi_j^{\rm in},\psi_i^{\rm in}),\ \ \ f_i=(f,\psi_i^{\rm in}),\ \ \ i,j=1,\cdots,N_{\rm in},\\
&&b_{j,i}=a(\psi_i^{\rm in},\psi_j^{\rm bd}),\ \ \ i=1,\cdots,N_{\rm in},\ j=1,\cdots,N_{\rm bd}.
\end{eqnarray}
Then the NN approximation of the original non-homogeneous problem is
\begin{eqnarray*}
u_{\rm NNh}=\sum_{i=1}^{N_{\rm in}}c^{\rm in}_i\psi_i^{\rm in}+\sum_{j=1}^{N_{\rm bd}}c^{\rm bd}_j\psi_j^{\rm bd}.
\end{eqnarray*}
\end{remark}
}

\comm{
\begin{remark}
The loss function in step 4 of Algorithm \ref{Algorithm_NNElement} can be any reasonable form derived from the target PDEs.
For numerical experiments of Section \ref{Section_Numerical}, 
we use the following energy functional discretized under the NN element basis
\begin{eqnarray*}
\mathcal L^{(\ell+1)}(c^{(\ell+1)},\theta^{(\ell)})=\frac{1}{2}(c^{(\ell+1)})^TA^{(\ell)}c^{(\ell+1)}-B^{(\ell)}c^{(\ell+1)}.
\end{eqnarray*}
A loss function based on a posterior error estimation like \cite{WangLinLiaoLiuXie} 
can also be used in the fourth step, which is easy to implement in the framework of NNEM.
We will have a more detailed discussion in future work.
\end{remark}
}

\section{Numerical examples}\label{Section_Numerical} 
In this section, we provide a numerical  example to validate the efficiency and accuracy 
of the proposed NN element method, Algorithm \ref{Algorithm_NNElement}. 
\comm{All the experiments are done on a NVIDA GeForce RTX 4090D GPU}. 

In order to show the %convergence behavior and 
accuracy of the Dirichlet boundary 
value problem, %and Neumann boundary value problem, 
we define the following errors for the NN element  solution $\Psi(x;c^*,\theta^*)$
\begin{eqnarray*}
e_{L^2}:=\|u-\Psi(x;c^*,\theta^*)\|_{L^2(\Omega)},
\ \ \ e_{H^1}:=\left|u-\Psi(x;c^*,\theta^*)\right|_{H^1(\Omega)},
\end{eqnarray*}
where $\|\cdot\|_{L^2}$ and $|\cdot|_{H^1}$ denote the $L^2(\Omega)$ norm and the $H^1(\Omega)$ seminorm, respectively.

%For the approximate eigenpair approximation $(\lambda^*,\Psi(x;c^*,\theta^*))$ by NN element method 
%for the eigenvalue problem, we define the $L^2(\Omega)$ projection operator $\mathcal P:H_0^1(\Omega)
%\rightarrow {\rm span}\{\Psi(x;c^*,\theta^*)\}$ as follows:
%\begin{eqnarray*}
%\left\langle\mathcal Pu,v\right\rangle_{L^2}=\left\langle u,v\right\rangle_{L^2}:=\int_\Omega uvdx,
%\ \ \ \forall v\in {\rm span}\{\Psi(x;c^*,\theta^*)\}\ \ {\rm for}\ u\in H_0^1(\Omega).
%\end{eqnarray*}
%And we define the $H^1(\Omega)$ projection operator $\mathcal Q:H_0^1(\Omega)\rightarrow {\rm span}\{\Psi(x;c^*,\theta^*)\}$
%as follows:
%\begin{eqnarray*}
%\left\langle\mathcal Qu,v\right\rangle_{H^1}=\left\langle u,v\right\rangle_{H^1}
%:=\int_\Omega\nabla u\cdot\nabla vdx,
%\ \ \ \forall v\in {\rm span}\{\Psi(x;c^*,\theta^*)\}\ \ {\rm for}\ u\in H_0^1(\Omega).
%\end{eqnarray*}
%Then we define the following errors for the approximated eigenvalue $\lambda^*$ and eigenfunction $\Psi(x;\theta^*)$
%\begin{eqnarray*}\label{relative_errors}
%e_\lambda:=\frac{|\lambda^*-\lambda|}{|\lambda|},\ \ \ e_{L^2}:=\frac{\|u-\mathcal Pu\|_{L^2(\Omega)}}{\|u\|_{L^2(\Omega)}},
%\ \ \ e_{H^1}:=\frac{\left|u-\mathcal Qu\right|_{H^1(\Omega)}}{\left|u\right|_{H^1(\Omega)}}.
%\end{eqnarray*}
%These relative errors are used to test the accuracy of the numerical examples in this section. 
%%for eigenvalue problems and PDEs.

In the implementation of the proposed NN element machine learning method,  
the neural networks are trained by Adam optimizer \cite{KingmaAdam} in combination with L-BFGS and the 
automatic differentiation in PyTorch is used to compute the derivatives.

%\subsection{Laplace problem with homogeneous Dirichlet boundary condition}
We consider the following Laplace problem with the homogeneous Dirichlet boundary condition: 
Find $u$ such that
\begin{eqnarray*}\label{Laplace}
\left\{
\begin{aligned}
-\Delta u&=f,\ \ \ &x\in&\Omega,\\
u&=0,\ \ \ &x\in&\partial\Omega,
\end{aligned}
\right.
\end{eqnarray*}
where $\Omega = (0,1)\times (0,1)$. % is the 2D unit square. 
%It is easy to know the exact solution is 
%$u(x)=\sum_{k=1}^d\sin(2\pi x_k)\cdot\prod_{i\neq k}^d\sin(\pi x_i)$.
%This is a well behaved problem with a smooth solution that has no trouble spots.
%It can be used for seeing how an adaptive algorithm behaves in a context where
%adaptivity isn’t really needed.
%Equation: Poisson
%Domain: unit square
%Boundary conditions: Dirichlet
Here, we set the exact solution as
$u(x,y)=\sin(\pi x)\sin(\pi y)$. %where the parameter $p$ determines the degree of the polynomial solution. 
%It should be chosen to be large enough such that the highest order finite elements to be used will not give the exact solution.
%The solution with $p = 10$ is shown in Figure 1, both as a color-mapped
%image and as a surface in perspective. The other figures that show solutions
%also present these two views.

%Notice that in this example, we set $b(x) = 0$. We employ 
Algorithm \ref{Algorithm_NNElement} is adopted to solve this problem with the 
following loss function 
%as well, but in Step 4, we use the following loss function instead:
\begin{eqnarray*}
\mathcal L^{(\ell+1)}(c^{(\ell+1)},\theta^{(\ell)}):= \frac{1}{2}\left\|\nabla \Psi(x;c^{(\ell+1)},\theta^{(\ell)})\right\|_0^2 
- \left(f,\Psi(x;c^{(\ell+1)},\theta^{(\ell)})\right), 
% \left\|f + \Delta\Psi(x;c^{(\ell+1)},\theta^{(\ell)})\right\|_0.
\end{eqnarray*}
for Step 4. 
%\comm{The corresponding a posteriori error estimate can be deduced 
%in the same way as Theorem \ref{Theorem_hDirichlet} 
%using Lemma \ref{lemma_Green}, Cauchy-Schwarz inequality 
%and Poincar\'e inequality. We omit the proof here}.

For each envelop function, \comm{the associated NN has two hidden layers 
and each hidden layer has $16$ neurons}
The activation function is chosen as the sine function in each hidden layer.
%\comm{We use the method described in Section \ref{section_alg_hDirichlet} 
%to guarantee the homogeneous boundary condition}.

\comm{In the computation of the integrations for the loss function, 
we choose 36 Gauss points in each triangular element $K\in\mathcal T_h$.
The Adam optimizer is employed with a learning rate $0.0003$ in the first $50000$ epochs
to produce the final NN element approximation.}

In order to show the efficiency of the proposed NN element method, 
we will also solve the problem (\ref{Laplace}) 
with the finite element methods. 
The corresponding numerical results are shown in Tables \ref{table_NNEMP2} and \ref{table_NNEMP3}.         
In Table \ref{table_NNEMP2}, FEMP2 denotes the second order finite element method is adopted to solve the 
associate problem, NNEMP2 means we use the basis of the second order Lagrange finite element space as the 
envelop functions to build the subspace $V_{\rm NNh}(\mathcal T_h, \theta^{(\ell)})$. 
The notation in Table  \ref{table_NNEMP3} has the similar definitions. 

Based on the results in Tables  \ref{table_NNEMP2} and \ref{table_NNEMP3}, 
we can find the NN element based machine learning method proposed in this paper 
shows good efficiency for solving (\ref{Laplace}).

%==================================================================
\begin{table}[htb!]
\begin{center}
\caption{Errors of homogeneous Dirichlet boundary value problem for the P2 NN element method.}\label{table_NNEMP2}
\begin{tabular}{ccc|ccc}
\hline
&   &   FEMP2&   &   NNEMP2\\
\hline
$h$&   $e_{H^1}$&   $e_{L^2}$&   $e_{H^1}$&   $e_{L^2}$\\
\hline
$\sqrt{2}/2$&      6.581e-01&   4.277e-02&      2.531e-02&   6.698e-04\\
$\sqrt{2}/4$&     1.831e-01&   5.559e-03&     2.181e-03&   4.339e-05\\
$\sqrt{2}/8$&     4.723e-02&   6.985e-04&     2.149e-04&   2.208e-06\\
$\sqrt{2}/16$&     1.191e-02&   8.741e-05&     &   \\
$\sqrt{2}/32$&     2.983e-03&   1.093e-05&     &   \\
\hline
\end{tabular}
\end{center}
\end{table}
%==================================================================                  
\begin{table}[htb!]
\begin{center}
\caption{Errors of homogeneous Dirichlet boundary value problem for the P3 NN element method.}\label{table_NNEMP3}
\begin{tabular}{ccc|ccc}
\hline
&   &   FEMP3&   &   NNEMP3\\
\hline
$h$&   $e_{H^1}$&   $e_{L^2}$&   $e_{H^1}$&   $e_{L^2}$\\
\hline
$\sqrt{2}/2$&     1.043e-01 &   1.022e-02&      9.005e-05&   2.026e-06\\
$\sqrt{2}/4$&     1.417e-02&   6.160e-04&     6.543e-05&   1.171e-06\\
$\sqrt{2}/8$&     1.778e-03&   3.626e-05&     4.670e-06&   3.945e-08\\
$\sqrt{2}/16$&     2.209e-04&   2.197e-06&     &   \\
$\sqrt{2}/32$&     2.749e-05&   1.353e-07&     &   \\
\hline
\end{tabular}
\end{center}
\end{table}

\section{Concluding remarks}
In this paper, we propose a type of NN element based machine learning method for solving 
partial differential equations. 
The method is designed by combining the finite element mesh and neural network 
to build a type of neural network element space.  
The finite element mesh provides a way to build the envelop functions to 
satisfy the boundary conditions directly on the complex geometric domains.  
The neural network and the corresponding machine learning method are 
adopted to improve the approximation accuracy of the 
associated NN element space. For understanding the proposed numerical method, 
we also proved the approximation error analysis based on the idea of the partition of unity. 
Numerical examples are also provided to validate the efficiency and accuracy of the 
proposed NN element based machine learning method. 

This paper only consider the triangle mesh on the two dimensional domains. It is easy to know 
that the method here can be extended more general type of meshes on the two or three dimensional 
domains. From this point of view, the method here has potential applications in 
the numerical simulation for engineering problems. 

We are only concerned with solving the homogeneous Dirichlet boundary condition 
with high accuracy and efficiency. It is obvious that the proposed numerical method 
can also handle the partial differential equations with other type of 
homogeneous and nonhomogeneous boundary conditions by using the way in \cite{WangLinLiaoLiuXie}.  
This paper takes the second order elliptic problem as the example 
to show the computing way of the proposed NN element based machine learning method. 
Of course, other type of problems can also be solved with the method here
which will be our future work. 

In Algorithm \ref{Algorithm_NNElement}, the Ritz type of loss function is adopted to update
the NN element space. It is easy to know that other types of loss functions, such as the a posteriori 
error estimation, can also be used here. 

Finally, we should point out that the method in this paper gives a way to build the engineering 
software for the machine learning methods which are designed to solve the problems from the 
engineering field.

\bibliographystyle{siamplain}
%\bibliography{references}

\newpage

\begin{center}
\Large Declaration of Competing Interest
\end{center}

%We would like to submit our manuscript entitled 
%``Solving High Dimensional Partial Differential Equations Using  
%Tensor Neural Network and A Posteriori Error Estimators'' for the 
%possible publication in ``SIAM Journal of Scientific Computing'' 
%if suitable.\\

We declare that we have no financial and personal relationships with other people 
or organizations that can inappropriately influence our work, 
there is no professional or other personal 
interest of any nature or kind in any product, service and/or company that could be 
construed as influencing the position presented in, 
or the review of, the manuscript entitled: Solving High-dimensional Partial 
Differential Equations Using  
Tensor Neural Network and A Posteriori Error Estimators.\\

No conflict of interest exits in the submission of this manuscript, and manuscript 
is approved by all authors for publication. 
I would like to declare on behalf of my co-authors that the work described was 
original research that has not been published previously, 
and not under consideration for publication elsewhere, in whole or in part.

\end{document}